\documentclass[letterpaper, 11pt,  reqno]{amsart}

\usepackage{amsmath,amssymb,amscd,amsthm,amsxtra, esint}

\usepackage[implicit=true]{hyperref}

\usepackage[left=32mm, right=32mm, 
bottom=27mm]{geometry}

\allowdisplaybreaks[2]

\usepackage{color}

\definecolor{gr}{rgb}   {0.,   0.69,   0.23 }
\definecolor{bl}{rgb}   {0.,   0.5,   1. }
\definecolor{mg}{rgb}   {0.85,  0.,    0.85}
\definecolor{yl}{rgb}   {0.8,  0.7,   0.}
\definecolor{or}{rgb}  {0.7,0.2,0.2}

\newtheorem{theorem}{Theorem} [section]

\newtheorem{lemma}[theorem]{Lemma}
\newtheorem{proposition}[theorem]{Proposition}
\newtheorem{remark}[theorem]{Remark}

\newtheorem*{acknowledgment}{Acknowledgments}

\newcommand{\noi}{\noindent}

\newcommand{\R}{\mathbb{R}}

\let\Re=\undefined\DeclareMathOperator*{\Re}{Re}
\let\Im=\undefined\DeclareMathOperator*{\Im}{Im}

\newcommand{\E}{\mathbb{E}}

\newcommand{\N}{\mathcal{N}}
\newcommand{\NB}{\mathbb{N}}

\newcommand{\F}{\mathcal{F}}

\newcommand{\be}{\beta}

\newcommand{\nb}{\nabla}

\newcommand{\Dl}{\Delta}
\newcommand{\eps}{\varepsilon}

\newcommand{\G}{\Gamma}
\newcommand{\ld}{\lambda}

\newcommand{\wt}{\widetilde}
\newcommand{\cj}{\overline}

\newcommand{\dt}{\partial_t}

\newcommand{\embeds}{\hookrightarrow}

\newcommand{\ta}{\theta}

\renewcommand{\o}{\omega}
\renewcommand{\O}{\Omega}

\newcommand{\les}{\lesssim}

\DeclareMathOperator{\Id}{Id}

\numberwithin{equation}{section}
\numberwithin{theorem}{section}

\newcommand{\too}{\longrightarrow}

\newcommand{\HS}{\textit{HS}\,}

\begin{document}
\baselineskip = 14pt

\title[On the mass-critical and energy-critical SNLS]
{On the stochastic nonlinear Schr\"{o}dinger equations  \\ at critical regularities}

\author[T. Oh]{Tadahiro Oh}
\address{
Tadahiro Oh\\
School of Mathematics\\
The University of Edinburgh\\
and The Maxwell Institute for the Mathematical Sciences\\
James Clerk Maxwell Building\\
The King's Buildings\\
 Peter Guthrie Tait Road\\
Edinburgh\\ 
EH9 3FD\\United Kingdom} 

\email{hiro.oh@ed.ac.uk}

\author[M. Okamoto]{Mamoru Okamoto}
\address{Mamoru Okamoto\\
Division of Mathematics and Physics, Faculty of Engineering, Shinshu University, 4-17-1 Wakasato, Nagano City 380-8553, Japan}
\email{m\_okamoto@shinshu-u.ac.jp}



\begin{abstract}
We consider the Cauchy problem for  the defocusing  stochastic nonlinear Schr\"odinger equations (SNLS) with an additive noise
in the mass-critical  and  energy-critical settings.
By adapting the probabilistic perturbation argument 
employed 
in the context of the random data Cauchy theory
 by the first author with B\'enyi and Pocovnicu
(2015)  
to the current stochastic PDE setting,  
we present a concise argument to establish  global well-posedness of the mass-critical and energy-critical SNLS.

\end{abstract}


\maketitle

%
%
%

\section{Introduction}\label{SEC:1}

\subsection{Stochastic nonlinear Schr\"odinger equations}

We consider the Cauchy problem for the  stochastic nonlinear Schr\"{o}dinger equation (SNLS)
with an additive noise:
\begin{equation}
\label{SNLS}
\begin{cases}
 i \dt u + \Dl  u = |u|^{p-1} u + \phi \xi \\
 u|_{t = 0} = u_0,
\end{cases}
\qquad (t, x) \in \R_+\times \R^d,
\end{equation}

\noi
where $\xi (t,x)$ denotes a space-time white noise on $\R_+ \times \R^d$ 
and $\phi$ is a bounded operator on $L^2(\R^d)$. 
In this paper, we restrict our attention to the defocusing case.
Our main goal is to present a concise argument in 
establishing global well-posedness of \eqref{SNLS}
in the so-called  {\it mass-critical} and {\it energy-critical} cases.

Let us first go over the notion of the scaling-critical regularity
for the  (deterministic) defocusing nonlinear Schr\"odinger equation (NLS):
\begin{equation}
i \partial_t u +   \Delta u =  |u|^{p-1} u,   
\label{NLS1}
\end{equation}

\noi
namely, \eqref{SNLS} with $\phi \equiv 0$.
The equation \eqref{NLS1}
is known to enjoy   the following dilation symmetry:
\begin{align*}
 u(t, x) \longmapsto u^\ld(t, x) = \ld^{-\frac{2}{p-1}} u (\ld^{-2}t, \ld^{-1}x)
\end{align*}

\noi
for  $\ld >0$.
If $u$ is a solution to \eqref{NLS1}, then the scaled function $u^\ld$ is also a solution to
\eqref{NLS1}
 with the rescaled initial data.
This dilation symmetry induces the following  scaling-critical Sobolev regularity:
 \begin{align*}
 s_\text{crit} = \frac d2 - \frac{ 2}{p-1}
 \end{align*}
 
 \noi
such that the homogeneous $\dot{H}^{s_\text{crit}}(\R^d)$-norm is invariant
under the dilation symmetry.
This critical regularity $s_\text{crit}$ provides
a threshold regularity for well-posedness and ill-posedness
of \eqref{NLS1}.
Indeed, 
when $s \geq \max(s_\text{crit}, 0)$, 
the Cauchy problem \eqref{NLS1} is known to be locally well-posed in $H^s(\R^d)$
\cite{GV79,  Kato, Tsu, CW}.\footnote
{When $p$ is not an odd integer, 
we may need to impose an extra assumption 
due to the non-smoothness of the nonlinearity.
}
On the other hand,
it is known that NLS \eqref{NLS1} is ill-posed
in the scaling supercritical regime:  $s < s_\text{crit}$.
See \cite{CCT, Kishimoto,  O17}.

Next, we introduce two important critical regularities
associated with the following conservation laws
for \eqref{NLS1}:
\begin{align*}
\text{Mass: }&  M(u(t)) := \int_{\R^d} |u(t, x)|^2 dx,\\
\text{Energy: } &  E(u(t)) := \frac 12 \int_{\R^d} |\nb u(t, x)|^2 dx
+ \frac{d-2}{2d} \int_{\R^d} |u(t, x)|^\frac{2d}{d-2} dx.
\end{align*}

\noi
In view of these conservation laws, 
we say that the equation \eqref{NLS1} is 

\begin{itemize}
\item[(i)]
{\it mass-critical}
when $s_\text{crit} = 0$,
namely, when $p = 1 + \frac 4d$,

\vspace{1mm}

\item[(ii)]
{\it energy-critical}
when $s_\text{crit} = 1$, 
namely, 
when $p = 1 + \frac 4{d-2}$ and $d\geq 3$.

\end{itemize}

\noi
Over the last two decades,
we have seen a significant progress 
in the global-in-time theory of
the defocusing NLS \eqref{NLS1} in the mass-critical and energy-critical cases
\cite{BO99, TVZ07, V, RycVis07, 
CKSTT, Dod12, 
 Dod161, Dod162}.
 In particular, we now know that 
 
 \begin{itemize}
 \item[(i)]
 the defocusing mass-critical
 NLS \eqref{NLS1} with $p = 1+ \frac 4d$ is globally well-posed
 in $L^2(\R^d)$,
 
\item[(ii)]  the defocusing energy-critical
 NLS \eqref{NLS1} with $p = 1+ \frac 4{d-2}$, $d \geq 3$,  is globally well-posed
 in $\dot H^1(\R^d)$.

 \end{itemize}
 
 \noi
 Moreover, 
 the following space-time bound  on a global solution $u$ to \eqref{NLS1} 
holds:
\begin{align}
\| u \|_{L_{t,x}^{\frac{2(d+2)}{d-2k}} (\R \times \R^d)} \le C(\| u_0 \|_{H^k}) < \infty
\label{bound}
\end{align}

\noi
with (i) $k = 0$ in the mass-critical case and (ii) $k = 1$ in the energy-critical case.
This bound in particular implies that  global-in-time solutions scatter, 
 i.e.~they  asymptotically behave like linear solutions 
 as $t \to \pm \infty$.

Let us now turn our attention to SNLS \eqref{SNLS}.
We say that $u$ is a solution to \eqref{SNLS} if it satisfies the following 
Duhamel formulation (= mild formulation):
\begin{align}
u(t) = S(t) u_0 -i \int_0^t S(t-t') |u|^{p-1}u(t') dt' -i \int_0^t S(t-t') \phi \xi (dt'), 
\label{SNLS2}
\end{align}

\noi
where  $S(t) = e^{it \Dl}$ denotes the linear Schr\"{o}dinger propagator.
The last term on the right-hand side of \eqref{SNLS2}
is called the stochastic convolution, which we denote by $\Psi$.
Fix  a probability space $(\O, \F, P)$ endowed with a filtration $\{ \F_t \}_{t \ge 0}$
and 
let $W$ denote the $L^2(\R^d)$-cylindrical Wiener process
associated with the filtration $\{ \F_t \}_{t \ge 0}$;
see \eqref{Psi1} below for a precise definition.
Then, the stochastic convolution $\Psi$ is defined by 
\begin{align}
\begin{split}
\Psi(t)  &   =  - i \int_0^t S(t - t') \phi \xi (dt')\\
:\! & =  - i \int_0^t S(t - t') \phi dW(t'). 
\end{split}
\label{sconv1}
\end{align}

\noi
See Section \ref{SEC:2} for the precise meaning of 
the definition \eqref{sconv1}; in particular see  \eqref{Psi2}.

Our main goal is to construct global-in-time dynamics 
for \eqref{SNLS2} in the mass-critical and energy-critical cases.
This means
that we take (i) $p = 1 + \frac 4d$ in the mass-critical case
and (ii) $p = 1 + \frac{4}{d-2}$ in the energy-critical case.
Furthermore, 
we take 
 the stochastic convolution $\Psi$  in \eqref{sconv1} 
 to be at the corresponding critical regularity.
Suppose that  $\phi \in \HS(L^2; H^s)$, 
namely, $\phi$ is a Hilbert-Schmidt
operator from $L^2(\R^d)$ to $H^s(\R^d)$.
Then, it is known that 
$\Psi \in C(\R_+; H^s(\R^d))$ almost surely; see \cite{DZ}.
Therefore, we will impose that 
 (i)   $\phi \in \HS(L^2; L^2)$ in the mass-critical case
and (ii) 
  $\phi \in \HS(L^2; H^1)$  in the energy-critical case.

Previously, 
de Bouard and Debussche \cite{dBD03}
studied SNLS \eqref{SNLS}
in the energy-subcritical setting: $s_\text{crit} < 1$, 
assuming that $\phi \in \HS(L^2;H^1)$.
By using  the Strichartz estimates,  
they showed that  the stochastic convolution $\Psi$
almost surely belongs to a right Strichartz space, 
which  allowed them to prove
local well-posedness of \eqref{SNLS} in $H^1(\R^d)$ 
with  $\phi \in \HS(L^2; H^1)$
 in the energy-subcritical case:
$1 < p < 1 + \frac{4}{d-2}$ when $d \geq 3$ and $1 < p < \infty$ when $d = 1, 2$.
We point out that 
when  $s \geq \max (s_\text{crit}, 0)$,
a slight modification of  the argument in \cite{dBD03}
with the regularity properties of the stochastic convolution (see Lemma~\ref{LEM:sconv} below)
yields
local well-posedness\footnote{When $p$ is not an odd integer, 
an extra assumption
such as $p \geq [s] + 1$ may be needed.}
 of \eqref{SNLS}
in $H^s(\R^d)$, 
provided that $\phi\in \HS(L^2; H^s)$.
See Lemma~\ref{LEM:LWPa}
for the statements in the mass-critical and energy-critical cases.
We also mention recent papers \cite{OPW, CPoc}
on local well-posedness of \eqref{SNLS}
with additive noises rougher than the critical regularities, 
i.e.~$\phi \in \HS(L^2; H^s)$ with $s < s_\text{crit}$.

In the energy-subcritical case, 
assuming $\phi \in \HS(L^2; H^1)$, 
global well-posedness of \eqref{SNLS} in $H^1(\R^d)$ 
 follows from 
an  a priori $H^1$-bound of solutions to \eqref{SNLS} based on the conservation of the energy $E(u)$
for the deterministic NLS
and Ito's lemma; see \cite{dBD03}.  See also Lemma~\ref{LEM:bound}. 
In a recent paper   \cite{CLO}, 
 Cheung, Li, and the first author
adapted the $I$-method \cite{CKSTT0} to the stochastic PDE setting
and 
established 
 global well-posedness of energy-subcritical SNLS below $H^1(\R^d)$. 
In the mass-subcritical case, 
global well-posedness in $L^2(\R^d)$
also follows  from 
an  a priori $L^2$-bound based on the conservation of the mass $M(u)$
for the deterministic NLS
and Ito's lemma.

We extend these global well-posedness results
to the mass-critical and energy-critical settings.

\begin{theorem} \label{THM:GWP}
\textup{(i) (mass-critical case).}
Let $d \geq 1$ and $p = 1+\frac4d$.
Then, given $\phi \in \HS (L^2; L^2)$, 
the defocusing mass-critical  SNLS \eqref{SNLS}
is globally well-posed in $L^2(\R^d)$.

\smallskip

\noi
\textup{(ii) (energy-critical case).}
Let $3 \leq d \leq 6$
and $p = 1 + \frac 4{d-2}$.
Then, 
given $\phi \in \HS(L^2; H^1)$, 
the defocusing energy-critical  SNLS \eqref{SNLS}
is globally well-posed in $H^1(\R^d)$.

\end{theorem}

In the following, we only consider deterministic initial data $u_0$.
This assumption is, however,  not essential and 
we may also take random initial data
(measurable with respect to the filtration $\F_0$ at time 0).

In the mass-critical case 
(and the energy-critical case, respectively), 
the a priori $L^2$-bound
(and the a priori $H^1$-bound, respectively)
 does not suffice for global well-posedness
(even in the case of the deterministic NLS \eqref{NLS1}).
The main idea for proving Theorem \ref{THM:GWP} 
is to adapt the probabilistic perturbation argument
introduced by the authors
 \cite{BOP2, OOP}
in studying global-in-time behavior of solutions 
to the defocusing energy-critical cubic NLS
with random initial data below the energy space.
Namely, 
by letting $v = u - \Psi$, where $\Psi$ is the stochastic convolution defined in \eqref{sconv1},
we study the equation satisfied by $v$:
\begin{equation}
\begin{cases}
i \partial_t v +  \Delta v 
=  \N(v+ \Psi)\\
v|_{t = 0} = u_0, 
\end{cases}
\label{SNLS3}
\end{equation}

\noi
where $\N(u) = |u|^{p-1} u$.
Write  the nonlinearity as
\[\N(v+ \Psi)
= \N(v)  + \big(\N(v+ \Psi) - \N(v)\big).
\]

\noi
Then, the regularity properties of the stochastic convolution 
(see Lemma \ref{LEM:sconv} below)
and the fact that their space-time norms can be made small on short time intervals
allow us to view the second term on the right-hand side
as a perturbative term.
By invoking the perturbation lemma (Lemmas~\ref{LEM:perturb}
and~\ref{LEM:perturbe}), 
we then compare the solution $v$ to \eqref{SNLS3}
with a solution to the deterministic NLS \eqref{NLS1}
on short time intervals
as in  \cite{BOP2, OOP}.
See also \cite{TVZ, KOPV}
for similar arguments in the deterministic case.
In the energy-critical case, 
we rely on the Lipschitz continuity of $\nabla \N(u) $ in the perturbation argument, 
which imposes the assumption $d \le 6$ in Theorem \ref{THM:GWP}.

\begin{remark}\label{REM:adapted} \rm
We remark that solutions constructed
in this paper are adapted to the given filtration 
 $\{ \F_t \}_{t \ge 0}$.
For example, adaptedness of a solution $v$ to \eqref{SNLS3} directly follows
from the local-in-time construction of the solution
via the Picard iteration.
Namely, we consider  the map $\G$ defined by
\begin{equation*}
\G v(t) := S(t) u_0 -i  \int_{0}^t S(t -t') \N (v+\Psi)(t') dt'.
\end{equation*}

\noi
Then, we define the $j$th Picard iterate $P_j$ by setting
\begin{align}
\begin{split}
P_1 & =  S(t) u_0, \\
 P_{j+1} & = \G P_j
 =  S(t) u_0 -i  \int_{0}^t S(t -t') \N (P_j+\Psi)(t') dt'
\end{split}
\label{adapt1}
\end{align}

\noi
for $j \in \NB$.
Since the stochastic convolution $\Psi$ is adapted to the filtration
 $\{ \F_t \}_{t \ge 0}$, 
it is easy to see from \eqref{adapt1}
that  $P_j$ is adapted for each $j \in \NB$.
Furthermore, 
the local well-posedness of~\eqref{SNLS3} by a contraction mapping principle
(see Lemmas  \ref{PROP:LWP2} and 
\ref{PROP:LWP2e} below)
shows that the sequence $\{P_j\}_{j \in \NB}$ 
converges, in appropriate functions spaces,
to a limit $v = \lim_{j \to \infty} P_j$, which is a solution 
to (the mild formulation of)
 \eqref{SNLS3}.
By invoking the closure property of measurability under a limit, 
we conclude that the solution $v$ to \eqref{SNLS3}
is also adapted to the filtration
 $\{ \F_t \}_{t \ge 0}$.
The same comment applies to 
Lemma \ref{LEM:LWPa} below.

\end{remark}

\begin{remark}\rm
(i) In the focusing case, i.e. with $- |u|^{p-1} u$ in \eqref{SNLS}, 
de Bouard and Debussche~\cite{dBD02} proved under appropriate conditions that, 
starting with any initial data, 
finite-time blowup occurs with positive probability.

\smallskip
\noi
(ii)  In the mass-subcritical and energy-critical 
cases,  SNLS with a multiplicative noise has been studied in \cite{BRZ14, BRZ16, BRZ17}.
In recent preprints, 
 Fan and Xu \cite{FanXu1} and Zhang \cite{Zha}
 proved global well-posedness of SNLS with a multiplicative noise
in the mass-critical and energy-critical setting.

%
%
%
%

\end{remark}

\section{Preliminary results} \label{SEC:2}

In this section, we introduce some notations
and go over preliminary results.

Given two separable Hilbert spaces $H$ and $K$, we denote by $\HS (H;K)$ the space of Hilbert-Schmidt operators $\phi$ from $H$ to $K$, endowed with the norm:
\[
\| \phi \|_{\HS(H;K)} = \bigg( \sum_{n \in \NB} \| \phi e_n \|_K^2 \bigg)^{\frac{1}{2}},
\]
where $\{ e_n \}_{n \in \NB}$ is an orthonormal basis of $H$.

Since our focus is the mass-critical and energy-critical cases, 
we introduce
$\N_k(u)$, $k = 0, 1$, by 
\begin{align}
\N_0(u) := |u|^{\frac{4}{d}}u
\qquad \text{and}\qquad
\N_1(u) := |u|^{\frac{4}{d-2}}u.
\label{nonlin}
\end{align}

\noi
Namely, $k = 0$ corresponds to the mass-critical case, 
while $k =1$ corresponds to the energy-critical case.

The  Strichartz estimates play an important role in our analysis.
We say that
a pair  $(q, r)$ is  admissible
if $2\leq q, r \leq \infty$, $(q, r, d) \ne (2, \infty, 2)$, and 
\begin{equation*}
\frac{2}{q} + \frac{d}{r} = \frac{d}{2}.
\label{Admin}
\end{equation*}

\noi
Then, the following Strichartz estimates
are known to hold; see \cite{Strichartz, Yajima, GV, KeelTao}.

\begin{lemma} \label{LEM:Str}
Let $(q,r)$ be admissible.
Then, we have
\[
\| S(t) \phi \|_{L_t^q L_x^r} \lesssim \| \phi \|_{L^2}.
\]

\noi
For any admissible pair $(\wt{q}, \wt{r})$, we also have
\begin{align}
\left\| \int_0^t S(t-t') F(t') dt' \right\|_{L_t^q L_x^r} \lesssim \| F \|_{L_t^{\wt{q}'} L_x^{\wt{r}'}},
\label{Str1}
\end{align}

\noi
where $\wt{q}'$ and $\wt{r}'$ denote the H\"older conjugates.
Moreover, if the right-hand side of \eqref{Str1} is finite for some admissible pair $(\wt q, \wt r)$, 
then $\int_0^t S(t-t') F(t') dt'$ is continuous (in time) with values in $L^2(\R^d)$.
\end{lemma}

Next, we provide a precise meaning to the stochastic convolution defined in \eqref{sconv1}.
Let $(\O, \F, P)$ be  a probability space  endowed with a filtration $\{ \F_t \}_{t \ge 0}$.
Fix an orthonormal basis $\{ e_n \}_{n \in \NB}$ of $L^2(\R^d)$.
We define an $L^2 (\R^d)$-cylindrical Wiener process $W$ by 
\begin{align}
W(t,x,\o) := \sum_{n \in \NB} \beta_n (t,\o)  e_n (x),
\label{Psi1}
\end{align}

\noi
where $\{ \beta_n \}_{n \in \NB}$ is a family of mutually independent complex-valued Brownian motions associated with the filtration $\{ \F_t \}_{t \ge 0}$.
Here, the complex-valued Brownian motion means that $\Re \beta_n(t)$ and $\Im \beta_n(t)$ are independent (real-valued) Brownian motions.
Then, the space-time white noise $\xi$ is given by a distributional derivative (in time) 
of $W$
and thus
we can express the stochastic convolution $\Psi$ as
\begin{align}
\Psi (t) = -i \sum_{n \in \NB} \int_0^t S(t-t') \phi e_n  \, d \beta_n (t'), 
\label{Psi2}
\end{align}

\noi
where each summand is a classical Wiener integral (with respect to the integrator
$d \be_n$); see \cite{Kuo}.
Then, we have the following lemma on the regularity properties of the stochastic convolution.
See, for example, Proposition 5.9 in \cite{DZ}
for Part (i).
As for Part (ii), see \cite{OPW}.

\begin{lemma} \label{LEM:sconv}
Let $d \geq 1$,  $T>0$, and $s \in \R$.
Suppose that  $\phi \in \HS (L^2; H^s)$.

\smallskip

\noi
\textup{(i)}
We have $\Psi \in C ( [0,T];H^s (\R^d))$ almost surely.
Moreover, for any finite $p  \ge 1$, there exists $C = C (T, p) >0$ such that
\[
\E \bigg[ \sup_{0 \le t \le T} \| \Psi (t) \|_{H^s}^p \bigg] 
\le C \| \phi \|_{\HS(L^2; H^s)}^p.
\]

\smallskip

\noi
\textup{(ii)}
Given $1 \le q < \infty$ and finite $r \ge 2$ such that $r \le \frac{2d}{d-2}$ when $d \ge 3$, we have $\Psi \in L^q ([0,T]; W^{s,r} (\R^d))$ almost surely.
Moreover,  for any finite $p \ge  1$, there exists $C = C (T, p) >0$ such that
\[
\E \bigg[  \| \Psi \|_{L^q([0, T]; W^{s,r}(\R^d))}^p \bigg] 
\le C \| \phi \|_{\HS(L^2;H^s)}^p.
\]
\end{lemma}

By the Strichartz estimates (Lemma \ref{LEM:Str})
and Lemma \ref{LEM:sconv} on the stochastic convolution, 
one can easily prove the following local well-posedness
(see Lemma \ref{LEM:LWPa} below)
of the mass-critical and energy-critical SNLS \eqref{SNLS}
by essentially following the argument in \cite{dBD03}, 
namely, 
by studying the Duhamel formulation for $v = u - \Psi$:
\[
v(t) = S(t) u_0 -i \int_0^t S(t-t') \N (v+ \Psi) (t') dt'.
\]

\noi
See also Lemmas~\ref{PROP:LWP2} 
and~\ref{PROP:LWP2e} below.
In the mass-critical case, 
the admissible pair
 $q=r=\frac{2(d+2)}{d}$ plays an important role.
In the energy-critical case, 
we use the following admissible pair
\begin{equation} \label{adqr}
(q_d, r_d) := \left( \frac{2d}{d-2}, \frac{2d^2}{d^2-2d+4} \right)
\end{equation}
for $d \ge 3$.

\begin{lemma} \label{LEM:LWPa}
\noi
\textup{(i)\,(mass-critical case).}
Let $d \geq 1$, $p = 1 + \frac 4d$,  and  $\phi \in\HS (L^2; L^2)$.
Then, 
given any  $u_0 \in L^2(\R^d)$, 
there exists an almost surely positive stopping time $T= T_{\o} (u_0) $
 and a unique local-in-time solution $u \in C([0,T]; L^2(\R^d))$ to the  mass-critical SNLS \eqref{SNLS}.
Furthermore, the following blowup alternative holds;
let $T^{\ast} = T^{\ast}_{\omega} (u_0)$ be the forward maximal time of existence.
Then, either
\[
T^{\ast} = \infty
\qquad \text{or} \qquad
\lim_{T \nearrow T^{\ast}} \| u \|_{L_{t,x}^{\frac{2(d+2)}{d}}([0,T) \times \R^d)} = \infty.
\]

\smallskip

\noi
\textup{(ii) (energy-critical case).}
Let  $3 \le d \le 6$, $p = 1 + \frac 4{d-2}$,   and $\phi \in\HS (L^2;H^1)$.
Then, given any  $u_0 \in H^1(\R^d)$, 
 there exists an almost surely positive  stopping time $T= T_{\o} (u_0) $  and 
 a unique local-in-time solution $u \in C([0,T]; H^1(\R^d))$ to the  energy-critical SNLS \eqref{SNLS}.
Furthermore,  the following blowup alternative holds;
let $T^{\ast} = T^{\ast}_{\omega} (u_0)$ be the forward maximal time of existence.
Then, either
\[
T^{\ast} = \infty
\qquad \text{or} \qquad
\lim_{T \nearrow T^{\ast}} \| u \|_{L^{q_d}([0,T);W^{1,r_d}(\R^d))} = \infty.
\]
\end{lemma}

We note that the mapping: $(u_0, \Psi) \mapsto v$ is continuous.
See Proposition 3.5 in \cite{dBD03}.
In the energy-critical case, the local-in-time well-posedness also holds for $d >6$ (see Remark \ref{REM:LWPe} below).
As mentioned earlier, 
 the perturbation argument requires the Lipschitz continuity of $\nabla \N$ 
 and hence we need to assume $d \le 6$ in the following.

Lastly, we state the a priori bounds on 
the mass and energy of 
solutions constructed in Lemma \ref{LEM:LWPa}.

\begin{lemma} \label{LEM:bound}
\textup{(i)\,(mass-critical case).}
Assume the hypotheses in  Lemma \ref{LEM:LWPa} (i).
Then, given $T_0>0$, 
there exists $C_1 = C_1 (M(u_0), T_0, \| \phi \|_{\HS (L^2; L^2)})>0$ such that for any stopping time $T$ with $0<T< \min (T^{\ast}, T_0)$ almost surely, we have
\begin{align}
\E \bigg[ \sup_{0\le t \le T} M(u(t)) \bigg] \le C_1,
\label{app1}
\end{align}

\noi
where $u$ is the solution to the  mass-critical SNLS \eqref{SNLS} with $u|_{t = 0}=u_0$
and 
 $T^{\ast} = T^{\ast}_{\omega} (u_0)$ is the forward maximal time of existence.

\smallskip

\noi
\textup{(ii) (energy-critical case).}
Assume the hypotheses in  Lemma \ref{LEM:LWPa} (ii).
Then, given $T_0 > 0$, there exists $C_2 = C_2 (M(u_0), E(u_0), T_0, \| \phi \|_{\HS (L^2;H^1)})>0$ such that for any stopping time $T$ with $0<T< \min (T^{\ast}, T_0)$ almost surely, we have
\[
\E \bigg[ \sup_{0\le t \le T} M(u (t)) \bigg] + \E \left[ \sup_{0\le t \le T} E(u (t)) \right] \le C_2,
\]
where $u$ is the solution to the defocusing energy-critical SNLS \eqref{SNLS} with $u|_{t = 0}=u_0$
and 
 $T^{\ast} = T^{\ast}_{\omega} (u_0)$ is the forward maximal time of existence.
\end{lemma}

For Part (ii),  we need to assume that the equation is defocusing.
These a priori bounds  follow
from Ito's lemma and the Burkholder-Davis-Gundy inequality.
In order to justify an application of Ito's lemma, 
one needs to go through a certain approximation argument.
See, for example,  Proposition 3.2 in \cite{dBD03}. 
In our mass-critical and energy-critical settings, however, 
such an approximation argument is more involved
and hence
we present a  sketch of  the argument in Appendix~\ref{SEC:A}.

\section{Mass-critical case}
\label{SEC:mass}

In this section, we prove global well-posedness
of  the defocusing mass-critical SNLS \eqref{SNLS} (Theorem \ref{THM:GWP}\,(i)).
 In Subsection \ref{SUBSEC:mass1}, 
 we first study the following 
 defocusing mass-critical 
  NLS with a deterministic perturbation:
\begin{equation} \label{ZNLS10}
i \dt v + \Dl v = \N_0 (v+f),
\end{equation}

\noi
where $\N_0$ is as in \eqref{nonlin} and $f$ is a given deterministic function,
 satisfying certain regularity conditions.
By applying the perturbation lemma, 
we prove global existence for \eqref{ZNLS10}, 
assuming an a priori $L^2$-bound of a solution $v$ to \eqref{ZNLS10}.
See Proposition \ref{PROP:perturb2}.
 In Subsection~\ref{SUBSEC:mass2},
 we then present the proof of  
 Theorem \ref{THM:GWP}\,(i)
 by writing \eqref{SNLS}
 in the form \eqref{ZNLS10}
 (with $f = \Psi$) 
 and verifying 
 the hypotheses in Proposition \ref{PROP:perturb2}.

\subsection{Mass-critical  NLS with a perturbation}
\label{SUBSEC:mass1}

By the standard Strichartz theory, 
we have the following local well-posedness of 
the perturbed NLS \eqref{ZNLS10}.

\begin{lemma}\label{PROP:LWP2}
There exists  small $\eta_0 >0$ such that
if
\begin{align*}
 \|S(t-t_0)v_0\|_{L_{t,x}^{\frac{2(d+2)}{d}} (I \times \R^d)} + \|f\|_{L_{t,x}^{\frac{2(d+2)}{d}} (I \times \R^d)} \leq \eta
\end{align*}

\noi
for some $\eta \leq \eta_0$ and
some time interval  $I = [t_0, t_1] \subset \R$,
then there exists a unique solution 
$v \in  C(I; L^2(\R^d))\cap L_{t, x}^{\frac{2(d+2)}{d}} (I \times \R^d)$
to \eqref{ZNLS10} with  $v(t_0) = v_0 \in L^2(\R^d)$.
Moreover, we have
\begin{align*}
\|v \|_{L_{t,x}^{\frac{2(d+2)}{d}} (I \times \R^d)} \le 2 \eta.
\end{align*}
\end{lemma}

\begin{proof}
We show that the map $\G$ defined by
\begin{equation*}
\G v(t) := S(t-t_0) v_0 -i  \int_{t_0}^t S(t -t') \N_0 (v+f)(t') dt'
\end{equation*}

\noi is a contraction on
the ball $B_{2\eta}
\subset L_{t, x}^{\frac{2(d+2)}{d}} (I \times \R^d)$ of radius $2\eta > 0$
centered at the origin, 
provided that $\eta > 0$ is sufficiently small.
Noting that 
 the H\"older conjugate of $\frac{2(d+2)}{d}$ is 
 $\frac{2(d+2)}{d+4} = \frac{2(d+2)}{d}/ (1 + \frac 4d)$, 
 it follows from Lemma \ref{LEM:Str} that 
 there exists small $\eta_0 > 0 $ such that 
\begin{align*}
 \|\G v\|_{L_{t,x}^{\frac{2(d+2)}{d}}(I \times \R^d)}
&  \leq
\|S(t-t_0) v_0\|_{L_{t,x}^{\frac{2(d+2)}{d}}(I \times \R^d)}
+  \|\Gamma v - S(t-t_0) v_0\|_{L_{t,x}^{\frac{2(d+2)}{d}}(I \times \R^d)} \notag \\
& \leq \eta + C \bigg( \| v \|_{L_{t,x}^{\frac{2(d+2)}{d}}(I \times \R^d)} + \| f \|_{L_{t,x}^{\frac{2(d+2)}{d}}(I \times \R^d)} \bigg)^{1+\frac{4}{d}} \notag \\
&\leq \eta + C \eta^{1+\frac{4}{d}}
\leq 2 \eta
\end{align*}

\noi
and
\begin{align*}
  \| \G v_1 - \G v_2 \|_{L_{t,x}^{\frac{2(d+2)}{d}}(I \times \R^d)}
  \leq \frac 12
\| v_1 - v_2 \|_{L_{t,x}^{\frac{2(d+2)}{d}}(I \times \R^d)}
\end{align*}

\noi
for any $v, v_1, v_2 \in B_{ 2\eta}$
and $0 < \eta \leq \eta_0$.
Hence,
$\G$ is a contraction on $B_{2\eta}$.	
Furthermore, 
 we have
\begin{align*}
  \| v\|_{L^{\infty}(I;L^2(\R^d))}
 & \leq
\|S(t-t_0)v_0\|_{L^{\infty}(I;L^2(\R^d))}\\
& \hphantom{X}
 + C \bigg( \| v \|_{L_{t,x}^{\frac{2(d+2)}{d}}(I \times \R^d)} + \| f \|_{L_{t,x}^{\frac{2(d+2)}{d}}(I \times \R^d)} \bigg)^{1+\frac{4}{d}} 
\notag \\
& \leq  \|v_0\|_{L^2} + C \eta^{1+\frac{4}{d}}
< \infty
\end{align*}

\noi
for any $v  \in B_{ 2\eta}$.
This shows that $v \in C(I; L^2(\R^d))$.
\end{proof}

Next,  we recall the long-time stability 
result in the mass-critical setting.
See \cite{TVZ} for the proof.

\begin{lemma}[Mass-critical perturbation lemma] \label{LEM:perturb}
Let $I$ be a compact interval.
Suppose that $v \in C(I; L^2(\R^d))$ satisfies the following perturbed NLS:
\begin{align}\label{PNLS1}
i \dt v + \Dl v = |v|^{\frac{4}{d}} v + e,
\end{align}
satisfying
\begin{equation*}
 \|v\|_{L^\infty(I; L^2(\R^d))}  + \| v\|_{L_{t,x}^{\frac{2(d+2)}{d}} (I \times \R^d)} \leq R
\end{equation*}
for some $R \geq 1$.
Then, there exists $\eps_0 = \eps_0(R) > 0$ such that if we have
\begin{align}
\|u_0 - v(t_0) \|_{L^2(\R^d)}
+
\|e\|_{L_{t,x}^{\frac{2(d+2)}{d+4}}(I \times \R^d)} \leq \eps
\label{PP2}
\end{align}

\noi
for
some  $u_0 \in L^2(\R^d)$,
some $t_0 \in I$, and some $\eps < \eps_0$, then
there exists a solution
$u \in C(I; L^2(\R^d))$
to the defocusing mass-critical NLS:
\begin{align}
i \dt u + \Dl u = |u|^{\frac{4}{d}}u
\label{NLS2}
\end{align}

\noi
with  $u(t_0) = u_0$
such that
\begin{align*}
\|u\|_{L^{\infty}(I; L^2(\R^d))}+\|u\|_{L_{t,x}^{\frac{2(d+2)}{d}} (I \times \R^d)} & \leq C_1(R), \notag \\
\|u - v\|_{L^{\infty}(I; L^2(\R^d))} + \| u-v \|_{L_{t,x}^{\frac{2(d+2)}{d}} (I \times \R^d)} & \leq C_1(R)\eps,
\end{align*}

\noi
where $C_1(R)$ is a non-decreasing function of $R$.
\end{lemma}

In the remaining part of this subsection,
we consider  long time existence of
solutions to the perturbed NLS \eqref{ZNLS10} 
under several assumptions.
Given $T>0$,
we assume that
there exist $ C, \ta > 0$  such that
\begin{equation}
 \|f\|_{L_{t,x}^{\frac{2(d+2)}{d}}(I\times \R^d)} \leq C |I|^{\ta}
\label{P0}
 \end{equation}

\noi
for any interval $I \subset [0, T]$.
Then,  Lemma \ref{PROP:LWP2}
guarantees
existence of a solution to the perturbed NLS \eqref{ZNLS10}, 
at least for a short time.
The following proposition establishes long time existence
under some hypotheses.

\begin{proposition}\label{PROP:perturb2}
Given $T>0$, assume the following conditions \textup{(i)} - \textup{(ii)}:
\begin{itemize}
\item[\textup{(i)}]
$f \in 
L_{t, x}^{\frac{2(d+2)}{d}} ([0, T] \times \R^d)$
satisfies \eqref{P0},

\smallskip

\item[\textup{(ii)}]
Given
a solution
$v$  to  \eqref{ZNLS10}, 
 the following a priori $L^2$-bound holds:
\begin{align}
\| v\|_{L^\infty([0, T]; L^2(\R^d))} \leq R
\label{P0a}
\end{align}

\noi
for some $R\geq 1$.
\end{itemize}

\noi
Then,
there exists
$\tau = \tau(R,  \ta)>0$
such that,  given any $t_0 \in [0, T)$,
a unique solution $v$ to~\eqref{ZNLS10} 
exists on $[t_0, t_0 + \tau]\cap [0, T]$.
In particular, the condition \textup{(ii)} guarantees existence 
of a unique solution $v$ to the perturbed NLS \eqref{ZNLS10} on
the entire interval $[0, T]$.
\end{proposition}

\begin{proof}
By setting $e = \N_0(v+f)  - \N_0(v)$, 
the equation \eqref{ZNLS10}
reduces to \eqref{PNLS1}.
In the following, we iteratively apply Lemma \ref{LEM:perturb}
on short intervals and
 show that there exists $\tau = \tau(R,   \ta) > 0$ such that
\eqref{PNLS1} is well-posed on $[t_0, t_0 + \tau] \cap [0, T] $ for any $t_0 \in [0, T)$.

Let $w$ be the global solution to the defocusing mass-critical NLS \eqref{NLS2} 
with $w(t_0) = v(t_0) = v_0$.
By the assumption \eqref{P0a}, we  have $\|w(t_0) \|_{L^2} \leq R$.
Then, by the space-time bound \eqref{bound}, we have 
\[\| w\|_{L^{\frac{2(d+2)}{d}}_{t, x}(\R \times \R^d)} \leq  C(R) < \infty.\]

\noi
Given small $\eta > 0$ (to be chosen later),
we  divide the interval $ [t_0, T]$
into $J = J(R,  \eta)\sim \big(C(R) /\eta\big)^\frac{2(d+2)}{d}$  many subintervals $I_j = [t_j, t_{j+1}]$
such that
\begin{align}
\|w\|_{L_{t,x}^{\frac{2(d+2)}{d}}(I_j \times \R^d)} \leq \eta.
\label{A1}
\end{align}

\noi
We point out that  $\eta$ will be  chosen as an absolute constant
and hence dependence of other constants on $\eta$ is not essential in the following.
Fix $\tau > 0$ (to be chosen later in terms of $R$ and $\ta$)
and write $[t_0, t_0+\tau] = \bigcup_{j = 0}^{J'} \big([t_0, t_0+\tau]\cap I_j\big)$
for some $J'  \leq J - 1$, where $[t_0, t_0+\tau]\cap I_j \ne \emptyset$ for $0 \leq j \leq J'$
and $[t_0, t_0+\tau]\cap I_j=\emptyset$ for $j > J'$.

Since the nonlinear evolution $w$ is small on each $I_j$,
it follows that the linear evolution $S(t-t_j) w(t_j)$ is also small on each $I_j$.
Indeed, from the Duhamel formula, we have
\[S(t-t_j) w(t_j) = w(t) - i \int_{t_j}^t S(t - t') \N_0(w)(t') dt'.\]

\noi
Then, by Lemma \ref{LEM:Str} and \eqref{A1}, we have 
\begin{align}
\|S(t-t_j) w(t_j) \|_{L_{t,x}^{\frac{2(d+2)}{d}}(I_j \times \R^d)}
&\leq \|w\|_{L_{t,x}^{\frac{2(d+2)}{d}}(I_j \times \R^d)} 
+ C  \|w\|_{L_{t,x}^{\frac{2(d+2)}{d}}(I_j \times \R^d)}^{1+ \frac{4}{d}}\notag  \\
&\leq \eta + C  \eta^{1+\frac{4}{d}}\notag \\
& \leq 2 \eta 
\label{P2}
\end{align}
	
\noi
for all $j = 0, \dots, J-1$,
provided that $\eta > 0$ is sufficiently small.

Now, we estimate $v$ on the first interval $I_0$.
By $v(t_0)=w(t_0)$ and \eqref{P2}, we have
\begin{align*}
\|S(t-t_0) v(t_0) \|_{L_{t,x}^{\frac{2(d+2)}{d}}(I_0 \times \R^d)}
= \|S(t-t_0) w(t_0) \|_{L_{t,x}^{\frac{2(d+2)}{d}}(I_0 \times \R^d)}
\leq 2 \eta.
\end{align*}
Let $\eta_0 > 0$ be  as in
Lemma \ref{PROP:LWP2}.
Then,
by the local theory (Lemma \ref{PROP:LWP2}),
we have
\begin{align*}
\|v\|_{L_{t,x}^{\frac{2(d+2)}{d}} (I_0 \times \R^d)}
\le 6 \eta,
\end{align*}

\noi
as long as
$3\eta < \eta_0$
and
 $\tau = \tau(\eta, \ta) = \tau(\ta)> 0$ is sufficiently small
so that
\begin{align}
\|f\|_{L_{t,x}^{\frac{2(d+2)}{d}} ([t_0, t_0+ \tau))} \leq C\tau^\ta \leq \eta.
\label{P4a}
\end{align}

Next, we estimate the error term.
By Lemma \ref{LEM:Str} and \eqref{P0}, we have
\begin{align}
\|e\|_{L_{t,x}^{\frac{2(d+2)}{d+4}}(I_0 \times \R^d)}
&\leq C \bigg( \| v \|_{L_{t,x}^{\frac{2(d+2)}{d}} (I_0 \times \R^d)} + \| f \|_{L_{t,x}^{\frac{2(d+2)}{d}} (I_0 \times \R^d)} \bigg)^{\frac{4}{d}} \| f \|_{L_{t,x}^{\frac{2(d+2)}{d}} (I_0 \times \R^d)} \notag \\
&\leq C \Big(\eta + \tau^\ta\Big)^\frac{4}{d}\tau^{\ta }\notag \\
&\leq C \tau^{\ta }
\label{P4b}
\end{align}

\noi
for any small $\eta, \tau > 0$.
Given $\eps > 0$, we can choose
 $\tau = \tau(  \eps, \ta)>0$ sufficiently small
 so that
\begin{align*}
\|e\|_{L_{t,x}^{\frac{2(d+2)}{d+4}}(I_0 \times \R^d)}
\leq  \eps.
\end{align*}

\noi
In particular, for $\eps < \eps_0$ with $\eps_0= \eps_0(R) > 0$
dictated by Lemma \ref{LEM:perturb},
the condition \eqref{PP2}
is satisfied on $I_0$.
Hence, by the perturbation lemma (Lemma \ref{LEM:perturb}), 
 we obtain
\begin{align*}
\| w - v\|_{L^{\infty} (I_0; L^2(\R^d))} + \| w-v \|_{L_{t,x}^{\frac{2(d+2)}{d}} (I_0 \times \R^d)} \leq C_1(R) \eps.
\end{align*}
	
\noi
In particular, we have
\begin{align}
\|w(t_1) - v(t_1) \|_{L^2(\R^d)}
\leq C_1(R) \eps.
\label{P6}
\end{align}

We now move onto the second interval $I_1$.
By   \eqref{P2} and Lemma \ref{LEM:Str} with \eqref{P6},
we have
\begin{align}
&\|S(t - t_1) v(t_1)\|_{L_{t,x}^{\frac{2(d+2)}{d}}(I_1 \times \R^d)} \notag \\
&\leq \|S(t-t_1)w(t_1)\|_{L_{t,x}^{\frac{2(d+2)}{d}}(I_1 \times \R^d)} + \|S(t-t_1)(w(t_1)-v(t_1))\|_{L_{t,x}^{\frac{2(d+2)}{d}}(I_1 \times \R^d)} \notag \\
&\leq 2 \eta + C_0\cdot C_1(R) \eps
\leq 3\eta
\label{A3}
\end{align}

\noi
by choosing $\eps = \eps(R, \eta) = \eps(R) > 0$ sufficiently small.
Proceeding as before,
it follows from
Lemma~\ref{PROP:LWP2}
with  \eqref{A3} that
\begin{align*}
\|v\|_{L_{t,x}^{\frac{2(d+2)}{d}} (I_1 \times \R^d)}
\leq 8 \eta,
\end{align*}

\noi
as long as $4\eta \leq \eta_0$
and
 $\tau > 0$ is sufficiently small
so that \eqref{P4a} is satisfied.
By repeating the computation in  \eqref{P4b} 
with  \eqref{P0}, we have
\begin{align*}
\|e\|_{L_{t,x}^{\frac{2(d+2)}{d+4}}(I_1 \times \R^d)}
\leq  C \tau^{\ta}\leq  \eps
\end{align*}

\noi
by choosing
 $\tau = \tau(\eps, \ta)>0$ sufficiently small.
Hence, by the perturbation lemma (Lemma~\ref{LEM:perturb}) applied
to the second interval $I_1$, 
we obtain
\begin{align*}
\| w - v\|_{L^{\infty}(I_1; L^2(\R^d))}  + \| w - v\|_{L_{t,x}^{\frac{2(d+2)}{d}}(I_1 \times \R^d)} 
\leq C_1(R) (C_1(R) + 1)\eps.
\end{align*}

\noi
provided that $\tau = \tau (\eps, \ta )>0$ is chosen sufficiently small
and that
$ (C_1(R) + 1)\eps <\eps_0$.
In particular, we have
\begin{align*}
\|w(t_2) - v(t_2) \|_{L^2(\R^d)}
\leq C_1(R)(C_1(R) + 1)\eps =: C_2(R)\eps.
\end{align*}

For $j \geq 2$, 
define 
$C_j(R)$  recursively by setting
\begin{align*}
C_j(R) = C_1(R)(C_{j-1}(R) + 1).
\end{align*}

\noi
Then,  proceeding inductively, 
we  obtain
\begin{align*}
 \|w(t_j) - v(t_j) \|_{L^2(\R^d)}
\leq C_j(R) \eps,
\end{align*}

\noi
for all $0 \leq j \leq J'$,
as long as 
 $\eps  = \eps(R, \eta, J) >0$ is sufficiently small
such that
\begin{itemize}
\item
  $C_0\cdot C_j(R) \eps \leq  \eta$
(here, $C_0$ is the constant from the Strichartz estimate in \eqref{A3}), 

\smallskip

\item
$(C_j(R)+1) \eps < \eps_0$,
\end{itemize}

\noi
for $j = 1, \dots, J'$.
Recalling that $J'+1 \leq J = J(R,  \eta)$,
we see that this can be achieved
by choosing  small $\eta >0$,
$\eps = \eps(R, \eta)  = \eps(R)> 0$,
and $\tau = \tau(\eps, \ta) = \tau(R, \ta) >0$ sufficiently small.
This guarantees existence of a (unique) solution $v$ to \eqref{ZNLS10} on $[t_0, t_0+\tau]$.
Lastly, noting that 
 $\tau > 0$ 
is independent of $t_0 \in [0, T)$, 
we conclude existence of the solution 
$v$ to~\eqref{ZNLS10} 
on the entire interval $[0, T]$.
\end{proof}

\subsection{Proof of Theorem \ref{THM:GWP}\,(i)}
\label{SUBSEC:mass2}

We are now ready to present a proof of Theorem \ref{THM:GWP}\,(i).
Given a local-in-time solution $u$ to \eqref{SNLS}, 
let  $v = u - \Psi$.
Then, $v$ satisfies 
\begin{equation}
\begin{cases}
i \partial_t v +  \Delta v 
=  \N_0(v+ \Psi)\\
v|_{t = 0} = u_0.
\end{cases}
\label{SNLS4}
\end{equation}

\noi
Theorem \ref{THM:GWP}\,(i)
follows from applying  Proposition \ref{PROP:perturb2}
to \eqref{SNLS4} with $f = \Psi$, 
once we verify the hypotheses (i) and (ii).

Fix $T > 0$. From Lemma \ref{LEM:bound}
and Markov's inequality, 
we have the following almost sure a priori bound:
\begin{align}
 \sup_{0\le t \le T} M(u(t)) \leq C 
 \big(\o, T, M(u_0),  \| \phi \|_{\HS (L^2; L^2)}\big) < \infty
\label{A4}
\end{align}

\noi
for a solution $u$ to \eqref{SNLS} with $p = 1 + \frac 4d$.
Then, from \eqref{A4} and Lemma \ref{LEM:sconv}\,(i), we obtain
\begin{align*}
 \sup_{0\le t \le T}M(v(t)) & =  \sup_{0\le t \le T}M(u(t) - \Psi(t))
\les  \sup_{0\le t \le T}M(u(t)) +  \sup_{0\le t \le T} M(\Psi(t))\\
& \leq C 
 \big(\o, T, M(u_0),  \| \phi \|_{\HS (L^2; L^2)}\big)< \infty
\end{align*}

\noi
almost surely.
This shows that the hypothesis (ii) in   Proposition \ref{PROP:perturb2}
holds almost surely for some almost surely finite $R = R(\o) \geq 1$.
The hypothesis (i) in   Proposition \ref{PROP:perturb2}
easily follows from 
H\"older's inequality in time, Markov's inequality, and 
Lemma \ref{LEM:sconv}\,(ii).
More precisely, 
by fixing finite $q >\frac{2(d+2)}{d}$
and noting 
$\frac{2(d+2)}{d} \leq \frac{2d}{d-2}$
for $d \geq 3$, 
Lemma \ref{LEM:sconv}\,(ii) yields
\[
\E \bigg[  \| \Psi \|_{L^q([0, T]; L^\frac{2(d+2)}{d}(\R^d))} \bigg] 
\le C \| \phi \|_{\HS(L^2;L^2)}.
\]

\noi
Then,  Markov's inequality yields
\begin{align}
\| \Psi \|_{L^q([0, T]; L^\frac{2(d+2)}{d}(\R^d))}
\leq C 
 \big(\o,   \| \phi \|_{\HS (L^2; L^2)}\big) < \infty, 
\label{Z1}
\end{align}

\noi
which in turn implies
$\Psi \in 
L_{t, x}^{\frac{2(d+2)}{d}} ([0, T] \times \R^d)$
almost surely.
Moreover, 
it follows
from~\eqref{Z1} and  H\"older's inequality in time 
that 
\[
\| \Psi \|_{L_{t, x}^{\frac{2(d+2)}{d}} (I \times \R^d)}
\leq |I|^\ta
\| \Psi \|_{L^q(I; L^\frac{2(d+2)}{d}(\R^d))}
\leq
C 
 \big(\o,   \| \phi \|_{\HS (L^2; L^2)}\big) |I|^\ta
\]

\noi
for any 
 interval $I \subset [0, T]$, 
 where $\ta = \frac{d}{2(d+2)} - \frac 1q  > 0$.
 This verifies
\eqref{P0}.

Hence, by applying 
  Proposition \ref{PROP:perturb2}, we can construct a solution $v$ to \eqref{SNLS4}
  on $[0, T]$.
  Since the choice of $T>0$ was arbitrary, 
this proves Theorem~\ref{THM:GWP}\,(i).

\section{Energy-critical case}

In this section, we prove global well-posedness
of  the defocusing energy-critical SNLS~\eqref{SNLS} (Theorem \ref{THM:GWP}\,(ii)).
The idea is to follow the argument for the mass-critical case presented in Section \ref{SEC:mass}.
Namely, we study the following 
 defocusing
 energy-critical 
  NLS with a deterministic perturbation:
\begin{equation} \label{ZNLS11}
i \dt v + \Dl v = \N_1 (v+f),
\end{equation}

\noi
where $\N_1$ is as in \eqref{nonlin} and $f$ is a given deterministic function,
 satisfying certain regularity conditions.

Let $q_d$ and $r_d$ be as in 
\eqref{adqr}
and set $\rho_d := \frac{2d^2}{(d-2)^2}$ for $d \ge 3$.
 A direct calculation shows that
\begin{equation} \label{est:energy1}
\frac{d+2}{d-2} \frac 1{q_d} = \frac{1}{q_d'}, \quad
\frac{1}{r_d'} = \frac{1}{r_d} + \frac{4}{d-2} \frac{1}{\rho_d}, \quad
\text{and}\quad
W^{1,r_d}(\R^d) \hookrightarrow L^{\rho_d}(\R^d).
\end{equation}

\subsection{Energy-critical   NLS with a perturbation}
\label{SUBSEC:energy1}

We first go over the local theory for the perturbed NLS \eqref{ZNLS11}
in the energy-critical case.

\begin{lemma}
\label{PROP:LWP2e}
Let $3 \le d \le 6$.
Then, there exists  small $\eta_0 = \eta_0>0$ such that
if
\begin{align}
 \|S(t-t_0)v_0\|_{L^{q_d} (I; W^{1,r_d} (\R^d))} + \|f\|_{L^{q_d} (I; W^{1,r_d} (\R^d))} \leq \eta
\label{Y1}
 \end{align}

\noi
for some $\eta \leq \eta_0$ and
some time interval  $I = [t_0, t_1] \subset \R$,
then there exists a unique solution $v \in C(I; H^1(\R^d))\cap L^{q_d} (I; W^{1,r_d}(\R^d))  $
to \eqref{ZNLS11}
with  $v(t_0) = v_0 \in H^1(\R^d)$.
Moreover, we have
\begin{align*}
\|v\|_{L^{q_d}(I; W^{1,r_d}(\R^d))} \le 2 \eta.
\end{align*}
\end{lemma}

\begin{proof}
We show that the map $\G$ defined by
\begin{equation*}
\G v(t) := S(t-t_0) v_0 -i  \int_{t_0}^t S(t -t') \N_1 (v+f)(t') dt'
\end{equation*}

\noi is a contraction on
$B_{2\eta} \subset L^{q_d} (I; W^{1,r_d}(\R^d))  $
of radius $2\eta > 0$ centered at the origin, 
provided that $\eta > 0$ is sufficiently small.
It   follows from Lemma \ref{LEM:Str} 
and \eqref{est:energy1} with \eqref{Y1}
that 
 there exists small $\eta_0 > 0 $ such that 
\begin{align*}
  \|\G v\|_{L^{q_d}(I ; W^{1,r_d}(\R^d))}
&  \leq
\|S(t-t_0) v_0\|_{L^{q_d}(I ; W^{1,r_d}(\R^d))}
+ C \| \N_1 (v+f) \|_{L^{q_d'}(I; W^{1,r_d'}(\R^d))} \notag \\
& \leq \eta+ C \bigg( \| v \|_{L^{q_d}(I; W^{1,r_d}(\R^d))} + \| f \|_{L^{q_d}(I; W^{1,r_d}(\R^d))} \bigg)^{1+ \frac 4{d-2}} \notag \\
&\leq \eta  + C \eta^{1+\frac{4}{d-2}}
\leq 2 \eta
\end{align*}
for $v \in B_{2\eta}$
and $0 < \eta \leq \eta_0$.
Recall that 
 $\nabla \N_1$ is Lipschitz continuous when $3\leq d \le 6$ 
 and we have
 \begin{align}
\nb N_1(u_1)  - \nb \N_1(u_2)
& = O\big(|u_1|^\frac{4}{d-2}+ |u_1|^\frac{4}{d-2}\big) \nb(u_1 - u_2)\notag\\
& \hphantom{X}
+ O\big((|u_1|^\frac{6-d}{d-2} + |u_1|^\frac{6-d}{d-2})|u_1 - u_2|\big) \nb u_2.
 \label{Y2}
 \end{align}
 
 \noi
 See, for example, Case 4 in the proof of Proposition 4.1 in \cite{OOP}.
 Then, proceeding as above with \eqref{Y2},  
  we have
\begin{align}
 \| \G & v_1 - \G v_2 \|_{L^{q_d}(I ; W^{1,r_d}(\R^d))} \notag \\
& \leq C \| \N_1 (v_1+f) - \N_1 (v_2+f) \|_{L^{q_d'}(I; W^{1,r_d'}(\R^d))} \notag \\
& \leq C \bigg( \| v_1 \|_{L^{q_d}(I; L^{\rho_d}(\R^d))} + \| v_2 \|_{L^{q_d}(I; L^{\rho_d}(\R^d))}
 + \| f \|_{L^{q_d}(I; L^{\rho_d}(\R^d))} \bigg)^{\frac{4}{d-2}} \notag \\
&\qquad \times \| v_1-v_2 \|_{L^{q_d}(I; W^{1,r_d}(\R^d))} \notag \\
&\quad + C \bigg( \| v_1 \|_{L^{q_d}(I; W^{1,r_d}(\R^d))} 
+ \| v_2 \|_{L^{q_d}(I; W^{1,r_d}(\R^d))} + \| f \|_{L^{q_d}(I; W^{1,r_d}(\R^d))} \bigg) \notag \\
&\qquad \times \bigg( \| v_1 \|_{L^{q_d}(I; L^{\rho_d}(\R^d))} 
+ \| v_2 \|_{L^{q_d}(I_0; L^{\rho_d}(\R^d))} + \| f \|_{L^{q_d}(I; L^{\rho_d}(\R^d))} \bigg)^{\frac{6-d}{d-2}} \notag \\
&\qquad \times \| v_1-v_2 \|_{L^{q_d} (I; L^{\rho_d}(\R^d))} \notag \\
&\leq C \eta^{\frac{4}{d-2}} \| v_1-v_2 \|_{L^{q_d}(I; W^{1,r_d}(\R^d))} \notag \\
&  \leq \frac 12
\| v_1 - v_2 \|_{L^{q_d}(I; W^{1,r_d}(\R^d))} \label{est:Gdiffe}
\end{align}
for $v_1, v_2 \in B_{ 2\eta}$ and $ 0 < \eta \leq \eta_0$.
Hence,
$\G$ is a contraction on $B_{2 \eta}$.
Furthermore, 
 we have
\begin{align*}
\| v\|_{L^{\infty}(I; H^1(\R^d))}
&\leq
 \|S(t-t_0)v_0\|_{L^{\infty}(I;H^1(\R^d))}
 + C \| \N_1 (v+f) \|_{L^{q_d'}(I; W^{1,r_d'}(\R^d))} \notag \\
& \leq \|v_0\|_{H^1} + C \eta^{1+\frac{4}{d-2}}
< \infty
\end{align*}

\noi
for $v \in B_{2\eta}$.
This shows that $v \in C(I; H^1(\R^d))$.
\end{proof}

\begin{remark} \label{REM:LWPe}\rm
The restriction $d \leq 6$ appears in 
\eqref{Y2} and  \eqref{est:Gdiffe}, 
where we used  the Lipschitz continuity of $\nabla \N_1$. 
Following  the  argument  in \cite{CW}, 
we can remove this restriction
and construct a solution 
 by carrying out a  contraction argument on  $B_{2\eta}
\subset L^{q_d} (I; W^{1,r_d}(\R^d))  $
equipped with the distance
\[ d(v_1, v_2) = 
\|v_1 - v_2\|_{L^{q_d}( I; L^{r_d}(\R^d))}.\]

\noi
Indeed, a slight modification of  the computation in \eqref{est:Gdiffe}
shows
$d(\G v_1,  \G v_2)
\leq \frac 12 d(v_1, v_2)$
for any $v_1, v_2 \in B_{2\eta}$.

\end{remark}

Next, we state 
 the long-time stability 
result in the energy-critical setting.
 See \cite{CKSTT,  TV, TVZ, KV}.
The following lemma is stated
in terms of  non-homogeneous spaces, 
the proof follows closely to that in the mass-critical case.

\begin{lemma}[Energy-critical perturbation lemma] \label{LEM:perturbe}
Let $3 \le d \le 6$ and $I$ be a compact interval.
Suppose that $v \in C(I; H^1(\R^d))$ satisfies the following perturbed NLS:
\begin{align*}
i \dt v + \Dl v = |v|^{\frac{4}{d-2}} v + e,
\end{align*}
satisfying
\begin{equation*}
\|v\|_{L^\infty(I; H^1(\R^d))} + \| v\|_{L^{q_d}(I; W^{1,r_d} (\R^d))}   \leq R
\end{equation*}
for some $R \geq 1$.
Then, there exists $\eps_0 = \eps_0(R) > 0$ such that if we have
\begin{align*}
\|u_0 - v(t_0) \|_{H^1(\R^d)}
+
\| e\|_{L^{q_d'}(I; W^{1,r_d'} (\R^d))} \leq \eps
\end{align*}

\noi
for
some  $u_0 \in H^1(\R^d)$,
some $t_0 \in I$, and some $\eps < \eps_0$, then
there exists a solution
$u \in C(I; H^1(\R^d))$
to the defocusing energy-critical NLS:
\begin{align*}
i \dt u + \Dl u = |u|^{\frac{4}{d-2}}u 
\end{align*}

\noi
with $u(t_0) = u_0$
such that
\begin{align*}
\|u\|_{L^{\infty}(I; H^1(\R^d))}+\|u\|_{L^{q_d}(I; W^{1,r_d} (\R^d))} & \leq C_1(R), \notag \\
\|u - v\|_{L^{\infty}(I; H^1(\R^d))} + \| u-v \|_{L^{q_d}(I; W^{1,r_d} (\R^d))} & \leq C_1(R)\eps,
\end{align*}

\noi
where $C_1(R)$ is a non-decreasing function of $R$.
\end{lemma}

With Lemmas \ref{PROP:LWP2e}
and \ref{LEM:perturbe}
in hand, 
we can repeat the argument in 
 Proposition \ref{PROP:perturb2}
 and obtain the following proposition.
 The proof is essentially identical 
 to that of  Proposition~\ref{PROP:perturb2}
 and hence we omit details.
We point out that,  
in applying the perturbation lemma (Lemma~\ref{LEM:perturbe})
with $e = \N_1(v+f)  - \N_1(v)$, 
we use \eqref{Y2}, which imposes the restriction $d \leq 6$.

\begin{proposition}\label{PROP:perturb2e}
Let $3 \le d \le 6$.
Given $T>0$, assume the following conditions \textup{(i)} - \textup{(ii)}:
\begin{itemize}
\item[\textup{(i)}]
$f \in  L^{q_d} ([0, T]; W^{1,r_d}(\R^d))$
and 
there exist $C, \ta > 0$  such that
\begin{equation*}
 \|f\|_{L^{q_d}(I; W^{1,r_d}(\R^d))} \leq C |I|^{\ta}
 \end{equation*}

\noi
for any interval $I \subset [0, T]$.

\item[\textup{(ii)}]
Given
a solution
$v$  to  \eqref{ZNLS11}, 
 the following a priori $H^1$-bound holds:
\begin{align*}
\| v\|_{L^\infty([0, T]; H^1(\R^d))} \leq R
\end{align*}

\noi
for some $R\geq 1$.
\end{itemize}

\noi
Then,
there exists
$\tau = \tau(R,\ta)>0$
such that,  given any $t_0 \in [0, T)$,
a unique solution $v$ to~\eqref{ZNLS11} with $k=1$
exists on $[t_0, t_0 + \tau]\cap [0, T]$.
In particular, the condition \textup{(ii)} guarantees existence 
of a unique solution $v$ to the perturbed NLS \eqref{ZNLS11} on
the entire interval $[0, T]$.

\end{proposition}

\subsection{Proof of Theorem \ref{THM:GWP}\,(ii)}
\label{SUBSEC:energy2}

As in Subsection \ref{SUBSEC:mass2}, 
Theorem \ref{THM:GWP}\,(ii)
follows from applying  Proposition \ref{PROP:perturb2e}
to \eqref{ZNLS11} with $f = \Psi$, 
once we verify the hypotheses (i) and (ii).

Fix $T>0$.
As in Subsection \ref{SUBSEC:mass2}, 
 the hypothesis (i) in   Proposition \ref{PROP:perturb2e}
can easily be verified  from 
H\"older's inequality in time, Markov's inequality, and 
Lemma \ref{LEM:sconv}\,(ii), 
once we note that
\[r_d = \frac{2d^2}{d^2-2d+4} 
\leq \frac{2d}{d-2}.\]

\noi
Furthermore, 
 the following almost sure a priori bound 
 follows from Lemma \ref{LEM:bound}
and Markov's inequality:
\begin{align}
 \sup_{0\le t \le T} \Big(M(u(t))+ E(u(t))\Big) \leq C 
 \big(\o, T, M(u_0), E(u_0),  \| \phi \|_{\HS (L^2; H^1)}\big) < \infty
\label{B4}
\end{align}

\noi
for a solution $u$ to \eqref{SNLS}
with $p = 1 + \frac 4 {d-2}$.
Then, from \eqref{B4} and Lemma \ref{LEM:sconv}\,(i), we obtain
\begin{align*}
 \sup_{0\le t \le T}\|v(t)\|_{H^1}
  & \leq   \sup_{0\le t \le T}\|u(t)\|_{H^1}
+    \sup_{0\le t \le T}\| \Psi(t)\|_{H^1}\\
& \leq C 
 \big(\o, T, M(u_0), E(u_0),  \| \phi \|_{\HS (L^2; H^1)}\big)< \infty
\end{align*}

\noi
almost surely.
This shows that the hypothesis (ii) in   Proposition \ref{PROP:perturb2e}
holds 
almost surely for some almost surely finite $R = R(\o) \geq 1$.
This proves Theorem~\ref{THM:GWP}\,(ii).

\appendix

\section{On the application of Ito's lemma}
\label{SEC:A}

In this appendix, we briefly 
discuss  the derivation of 
the a priori bounds 
on the mass and the energy 
stated in 
Lemma \ref{LEM:bound}.
The argument essentially follows
from that by de Bouard-Debussche \cite{dBD03} 
but we indicate certain required modifications.

\subsection{Mass-critical case} \label{SUBSEC:A1}
We first consider the mass-critical case.
Given $N \in \NB$, let 
$P_{ N}$ denote
a smooth frequency projection onto $\{ |\xi| \les N \}$ and 
set $\phi_N := P_{N} \circ \phi$.
Then, consider
 the following truncated SNLS:
\begin{equation}
\label{SNLS5a}
\begin{cases}
i \dt u_N + \Dl u_N =  \N_0 (u_N) + \phi_N \xi \\
u_N|_{t = 0} = P_{N} u_0,
\end{cases}
\end{equation}

\noi
where $\N_0$ is as in \eqref{nonlin}.
Note that  $P_N u_0 \in H^1(\R^d)$
and $\phi_N \in \HS(L^2; H^1)$.
Therefore, it follows from \cite{dBD03}
that 
\eqref{SNLS5a} is globally well-posed for each $N \in \NB$.
Furthermore, from Proposition 3.2 in \cite{dBD03},
we have 
\begin{align}
M(u_N(t))
= M( P_{N} u_0) + 2 \Im \sum_{n \in \NB} \int_0^t \int_{\R^d}
 \cj{u_N(t')} \phi_N e_n dx d \beta_n (t') + 2t \| \phi_N \|_{\HS (L^2; L^2)}^2
\label{SNLS6}
\end{align}

\noi
for any $t\geq 0$
and,  
as a consequence of \eqref{SNLS6}
and the Burkholder-Davis-Gundy inequality
(see, for example, \cite[Theorem 3.28 on p.\,166]{KS}), 
 the a priori bound \eqref{app1}
holds for each $u_N$, 
with the constant $C_1$,  independent of $N \in \NB$.

%
%
%

Given $T_0 > 0$, 
let $0 < T < \min (T^*, T_0)$ be 
a given stopping time as in Lemma \ref{LEM:bound}\,(i)
and 
$u$ be 
the solution
to \eqref{SNLS} constructed in Lemma \ref{LEM:LWPa}\,(i).
We now show that 
the solution  $u_N$ 
to the truncated SNLS \eqref{SNLS5a}
converges to $u$ almost surely.  
Then,  the a priori bound \eqref{app1}
for $u$ 
follows from that for $u_N$ mentioned above
and the convergence of $u_N$ to~$u$.

In the following, we suppress the spatial domain $\R^d$ for 
simplicity of the presentation.
Given $R > 0$, define a stopping time $T_1$ by setting
\[ T_1 = T_1(R): = \inf \Big\{ \tau \geq 0: \| u \|_{L_{t,x}^{\frac{2(d+2)}{d}}([0, \tau] )}
\geq R\Big\} \]

\noi
and set $T_2 : = \min(T, T_1)$.
In view of the  blowup alternative in Lemma \ref{LEM:LWPa}, 
we have
\begin{equation*} 
\| u \|_{L_{t,x}^{\frac{2(d+2)}{d}}([0,T] )} < \infty
\end{equation*}

\noi
almost surely
and hence we conclude that $T_2 \nearrow T$ almost surely as $R \to \infty$.

Given small $\eta > 0$ (to be chosen later),
we divide the interval $[0, T_2]$
into $J = J(R, \eta)$ many random subintervals $I_j = I_j(\o) = [t_j, t_{j+1}]$
with $t_0 = 0 < t_1 < \cdots < t_{J-1} < t_{J} = T_2$
such that
\begin{equation} \label{divideint}
\|u\|_{L_{t,x}^{\frac{2(d+2)}{d}}(I_j)} 
 \sim \eta
\end{equation}

\noi
for $j = 0, 1, \dots, J-1$.

Define the truncated stochastic convolution $\Psi_N$ by 
\[
\Psi_N (t) := -i \sum_{n \in \NB} \int_0^t S(t-t') \phi_N e_n d \beta_n (t')
\]

\noi
and set
\begin{align}
C_N^{(j)}(\o, u_0, \phi)
& = \| u(t_j) - u_N(t_j) \|_{L^2} \notag \\
& \hphantom{X}
+ \| \Psi - \Psi_N \|_{L_{t,x}^{\frac{2(d+2)}{d}} ([0, T_2])} 
+ \| \Psi - \Psi_N \|_{L^{\infty} ([0, T_2];L^2)}
\label{O1a}
\end{align}

\noi
for $j = 0, 1, \dots, J-1$.
Then,
it follows from 
 the Lebesgue dominated convergence theorem
(applied to $(\Id - P_N)u_0$)
and  Lemma \ref{LEM:sconv} that 
\begin{align}
 C_N^{(0)}(\o, u_0, \phi)
\too 0 
\label{O1b}
\end{align}

\noi
almost surely as $N \to \infty$.

From the Strichartz estimates
(Lemma \ref{LEM:Str}), we have
\begin{align}
 \| u -  u_N  & \|_{L_{t,x}^{\frac{2(d+2)}{d}} ([0, \tau])} 
\lesssim \| u(0) - u_N(0) \|_{L^2} 
\notag   \\
&\quad 
+ \bigg( \| u \|_{L_{t,x}^{\frac{2(d+2)}{d}}([0, \tau])} 
+ \| u_N \|_{L_{t,x}^{\frac{2(d+2)}{d}}([0, \tau])} \bigg)^{\frac{4}{d}} 
\| u - u_N \|_{L_{t,x}^{\frac{2(d+2)}{d}}([0, \tau])} \notag \\
&\quad 
+ \| \Psi - \Psi_N \|_{L_{t,x}^{\frac{2(d+2)}{d}} ([0, \tau])} 
\label{O1}
\end{align}

\noi
for any subinterval $[0, \tau] \subset I_0 = [0, t_1]$.
Then, from \eqref{O1} with 
  \eqref{divideint} and \eqref{O1a},
  we obtain 
\begin{align}
 \| u - u_N \|_{L_{t,x}^{\frac{2(d+2)}{d}} ([0, \tau])} 
&\lesssim C_N^{(0)}(\o, u_0, \phi)
\notag   \\
&\quad 
+ \bigg( \eta 
+ \| u_N \|_{L_{t,x}^{\frac{2(d+2)}{d}}([0, \tau])} \bigg)^{\frac{4}{d}} 
\| u - u_N \|_{L_{t,x}^{\frac{2(d+2)}{d}}([0, \tau])} 
\label{O2}
\end{align}

\noi
for any $0 \leq \tau \leq t_1$.
By taking $\eta > 0$ sufficiently small, 
 a standard continuity argument 
with  \eqref{O2}
and \eqref{O1b}
yields
\begin{equation} 
\|u_N\|_{L_{t,x}^{\frac{2(d+2)}{d}}(I_0)} 
 \sim \eta, 
 \label{O3}
\end{equation}

\noi
uniformly in $N \geq N_0(\o)$.
Applying Lemma \ref{LEM:Str} once again
with \eqref{divideint} and \eqref{O3}, 
we then have 
\begin{align*}
\| u - u_N \|_{L^{\infty} (I_0;L^2)} 
& +  \| u - u_N \|_{L_{t,x}^{\frac{2(d+2)}{d}} (I_0)} \\
&\lesssim \| u(0) - u_N(0) \|_{L^2} 
+ \eta^{\frac{4}{d}} 
\| u - u_N \|_{L_{t,x}^{\frac{2(d+2)}{d}}(I_0)} \notag \\
&\quad 
+ \| \Psi - \Psi_N \|_{L^{\infty} (I_0;L^2)}
+ \| \Psi - \Psi_N \|_{L_{t,x}^{\frac{2(d+2)}{d}} (I_0)} 
\end{align*}

\noi
uniformly in $N \geq N_0(\o)$.
Thus, from \eqref{O1a} and \eqref{O1b}, we conclude that 
\begin{align*}
\| u - u_N \|_{L^{\infty} (I_0;L^2)} 
&\les C_N^{(0)}(\o, u_0, \phi) \too 0 
\end{align*}

\noi
as $N \to \infty$.
In particular, we have
\begin{align}
\| u (t_1) - u_N (t_1) \|_{L^2} 
&\les C_N^{(0)}(\o, u_0, \phi) \too 0 
\label{O4}
\end{align}

\noi
as $N \to \infty$.
By repeating the argument above, 
we have
\begin{equation*} 
\|u_N\|_{L_{t,x}^{\frac{2(d+2)}{d}}(I_1)} 
 \sim \eta, 
\end{equation*}

\noi
uniformly in $N \geq N_1(\o)$.
Together with \eqref{O4}, this 
 yields
\begin{align*}
\| u - u_N \|_{L^{\infty} (I_1;L^2)} 
&\les C_N^{(1)}(\o, u_0, \phi) \too 0 
\end{align*}

\noi
as $N \to \infty$.

By successively applying the argument above
to the interval $I_j$, $j = 0, 1, \dots, J-1$, 
we conclude that 
\begin{align*}
\| u - u_N \|_{L^{\infty} (I_j;L^2)} 
&\les C_N^{(j)}(\o, u_0, \phi) \too 0 
\end{align*}

\noi
as $N\to \infty$.
Therefore, recalling that $J= J(R, \eta)$ depends only on $R>0$ and an absolute constant $\eta > 0$, 
we obtain 
\begin{align*}
\| u - u_N \|_{L^{\infty} ([0, T_2] ;L^2)} 
&\les \sum_{j = 0}^{J-1}C_N^{(j)}(\o, u_0, \phi) \too 0. 
\end{align*}

\noi
By the almost sure convergence of $u_N$ to $u$ in $C([0, T_2]; L^2(\R^d))$, 
 Fatou's lemma,
 and the uniform bound  \eqref{app1} for $u_N$,  we then have
\begin{align*}
\E \bigg[ \sup_{0\le t \le T_2} M(u(t)) \bigg] 
& = \E \bigg[ \lim_{N \to \infty} \sup_{0\le t \le T_2} M(u_N(t)) \bigg]\notag \\
& \leq \lim_{N \to \infty} \E \bigg[  \sup_{0\le t \le T_2} M(u_N(t)) \bigg]
\le C_1.
\end{align*}

\noi
Finally, 
from the almost sure convergence
of  $T_2 = T_2(R)$ to $T$, as $R\to \infty$, 
and Fatou's lemma, we conclude the bound \eqref{app1}
for $u$.
This proves Lemma \ref{LEM:bound}\,(i).

\subsection{Energy-critical case}
Next, we consider the energy-critical case.
In the following, we only discuss the 
a priori bound on the energy:
\begin{align}
 \E \left[ \sup_{0\le t \le T} E(u (t)) \right] \le C_2,
\label{app2}
\end{align}

\noi
since the a priori bound on the mass follows
in a similar but simpler manner.

\begin{lemma}\label{LEM:ito}
Assume the hypotheses in Lemma \ref{LEM:LWPa} \textup{(ii)}.
Then, for any stopping time $T$ such that $0 < T<T^\ast$ almost surely,
we have
\begin{align}
E(u(T))
&= E( u_0) - \Im \sum_{n \in \NB} \int_0^T \int_{\R^d}
\cj{(\Delta u -  \N_1(u)) (t')} \phi e_n dx d \beta_n (t') \notag\\
&\quad + \frac{d}{d-2} \sum_{n \in \NB} 
\int_0^T \int_{\R^d} |u(t')|^{\frac{4}{d-2}} | \phi e_n |^2 dx dt' + T \| \phi \|_{\HS (L^2; \dot H^1)}^2,
\label{E1}
\end{align}
where $u$ is the solution to the  energy-critical SNLS \eqref{SNLS}
with $p = 1 + \frac{4}{d-2}$, 
$\N_1$ is as in~\eqref{nonlin}, 
and $T^*$ is the forward maximal time of existence.
\end{lemma}

Once we prove Lemma \ref{LEM:ito}, 
the bound \eqref{app2} follows
from the Burkholder-Davis-Gundy inequality.

\begin{proof}
A  direct calculation shows that
\begin{align*}
E'(u(t)) (v) 
&=  \Re \int_{\R^d} \Big( \nabla u(t) \cdot \cj{\nabla v} + 
\N_1(u) (t) \cj{v} \Big) dx, \\
E'' (u(t)) (v_1,v_2)
&= \Re \int_{\R^d} \nabla v_1 \cdot \cj{\nabla v_2 } dx
+ \Re \int_{\R^d} |u(t)|^{\frac 4{d-2}} v_1 \cj{v_2} dx \\
&\quad+ \frac 4{d-2} \int_{\R^d} |u(t)|^{\frac 4{d-2}-2}
\Re \big( u(t) \cj{v_1} \big) \Re \big( u(t) \cj{v_2} \big) dx
\end{align*}
for $v, v_1, v_2 \in H^1(\R^d)$.
Thus,  
a formal application of Ito's lemma to $E(u(t))$ yields \eqref{E1}.
It remains to justify the application of Ito's lemma.

As in the proof of Proposition 3.3 in \cite{dBD03},
given $N \in \NB$,
we consider the following truncated problem:
\begin{equation}
\label{SNLS5}
\begin{cases}
 i \dt u_N  + \Dl u_N = P_N \N_1(u_N)
+ \phi_N \xi \\
 u_N|_{t = 0} = P_N u_0,
\end{cases}
\qquad (t, x) \in \R_+\times \R^d,
\end{equation}

\noi
where $P_N$ and $\phi_N$ are the same as those in Subsection \ref{SUBSEC:A1}.
Since the frequency truncation is harmless, the same well-posedness result as in Lemma \ref{LEM:LWPa} (ii) holds
for the truncated SNLS~\eqref{SNLS5}.
Moreover, by considering the corresponding Duhamel formulation  for \eqref{SNLS5}, we have $u_N = P_{3N} u_N$.
We can therefore apply Ito's lemma (see Theorem 4.32 in \cite{DZ}) to $E(u_N(t))$ and obtain
\begin{align}
E(u_N(t))
&= E(P_N u_0)
-\Im \sum_{n \in \NB} \int_0^t \int_{\R^d}
\cj{( \Delta u_N - \N_1( u_N) )(t')} \phi_N e_n dx d \beta_n (t') \notag\\
&\quad
+ \Im \int_0^t \int_{\R^d}
\cj{\Delta u_N(t') }
( \Id - P_N )\N_1( u_N)  (t') dx dt' \notag\\
&\quad + \frac{d}{d-2} \sum_{n \in \NB} \int_0^t \int_{\R^d} |u_N(t')|^{\frac{4}{d-2}} | \phi_N e_n |^2 dx dt' + t \| \phi_N \|_{\HS(L^2; \dot{H}^1)}^2
\label{E3}
\end{align}
 for $0< t< T^\ast_N$,
where $T^\ast_N$ is the forward maximal time of existence for 
the solution $u_N$ to~\eqref{SNLS5}.

Given $R > 0$, define a stopping $T_1$ by setting
\[ T_1 = T_1(R): = \inf \Big\{ \tau \geq 0: \| u \|_{L^{q_d} ([0,\tau]; W^{1,r_d})}
\geq R\Big\} \]

\noi
and set $T_2 : = \min(T, T_1)$, 
where $T$ is the stopping time given in Lemma \ref{LEM:bound}\,(ii)
with $0 < T < \min (T^*, T_0)$.
In view of the  blowup alternative in Lemma \ref{LEM:LWPa}, 
we have
\begin{equation} \label{blowal2}
\| u \|_{L^{q_d} ([0,T]; W^{1,r_d})} < \infty
\end{equation}

\noi
almost surely
and hence we conclude that $T_2 \nearrow T$ almost surely as $R \to \infty$.

From \eqref{est:energy1} and \eqref{Y2}, we have
\begin{align}
\| \N_1 & (u) - P_N \N_1(u_N) \|_{L^{q_d'} (I; W^{1,r_d'})} \notag\\
&\les  \| (\Id-P_N) \N_1(u) \|_{L^{q_d'} (I; W^{1,r_d'})}        + \| \N_1(u) - \N_1 (u_N) \|_{L^{q_d'} (I; W^{1,r_d'})} \notag\\
&\les
\| (\Id-P_N) \N_1(u) \|_{L^{q_d'} (I; W^{1,r_d'})} \notag\\
&\hphantom{X} + 
\Big( \| u \|_{L^{q_d}(I; W^{1,r_d})} + \| u_N \|_{L^{q_d}(I; W^{1,r_d})} \Big)^{\frac 4{d-2}}
\| u - u_N \|_{L^{q_d}(I; W^{1,r_d})}
\label{D2}
\end{align}

\noi
for any interval $I \subset [0,T_2]$.
It follows from the Lebesgue dominated convergence theorem and~\eqref{blowal2}
that the first term on the right-hand side of \eqref{D2} converges to $0$ almost surely as $N \to \infty$.
Accordingly, proceeding as in  Subsection \ref{SUBSEC:A1}, 
we conclude that 
 $u_N$ converges to $u$ in $C([0, T_2]; H^1(\R^d)) \cap L^{q_d}([0,T_2]; W^{1,r_d}(\R^d))$
 almost surely.
In particular, 
there exists an almost surely finite number $N_0 (\o) \in \NB$ such that 
$T^\ast_N \ge T_2$ for any $N \geq N_0(\o)$
and, as a result, 
 \eqref{E3} holds for any $0 <t < T_2$
 and $N \geq N_0(\o)$.
Moreover, from the definition of~$T_2 = T_2(R)$, we may assume  
 \begin{equation} \label{blowal3}
\| u_N \|_{L^{q_d} ([0,T_2]; W^{1,r_d})} \leq R + 1
\end{equation}

\noi
for any $N \geq N_0(\o)$.

This allows us to conclude that 
 the third term on the right-hand side of \eqref{E3} tends to~$0$ almost surely as $N \to \infty$.
Indeed,
by \eqref{est:energy1}, \eqref{Y2},  \eqref{blowal2}, 
 \eqref{blowal3},
 and the almost sure convergence of $u_N$ to $u$
 in $L^{q_d}([0,T_2]; W^{1,r_d}(\R^d))$, 
we have, for any $0 \leq t \leq T_2$, 
\begin{align}
\bigg|
\int_0^t  & \int_{\R^d}
\cj{\Delta u_N(t') }
( \Id - P_N ) \N_1(u_N ) (t') dx dt' \bigg|  \notag \\
&\le
\bigg|
\int_0^t \int_{\R^d}
( \Id - P_N ) 
\cj{\Delta u(t')} \cdot \N_1 (u) (t') dx dt' \bigg| \notag \\
&\quad+
\bigg|
\int_0^t \int_{\R^d}
\cj{(\Delta u - \Delta u_N)(t')} 
( \Id - P_N ) \N_1(u) (t') dx dt' \bigg| \notag \\
&\quad+
\bigg|
\int_0^t \int_{\R^d}
\cj{\Delta u_N(t') }
( \Id - P_N ) \big( \N_1 (u) - \N_1(u_N) \big) (t') dx dt' \bigg| \notag \\
&\les
\| (\Id-P_N) u \|_{L^{q_d}([0,T_2]; W^{1,r_d})}
\| \N_1(u) \|_{L^{q_d'}([0,T_2]; W^{1,r_d'})} \notag \\
&\quad+
\| u - u_N \|_{L^{q_d}([0,T_2]; W^{1,r_d})}
\| \N_1(u) \|_{L^{q_d'}([0,T_2]; W^{1,r_d'})} \notag \\
&\quad+
\| u_N \|_{L^{q_d}([0,T_2]; W^{1,r_d})}
\| \N_1(u) - \N_1(u_N) \|_{L^{q_d'}([0,T_2]; W^{1,r_d'})} \notag \\
&\les
\| (\Id-P_N) u \|_{L^{q_d}([0,T_2]; W^{1,r_d})}
\| u \|_{L^{q_d}([0,T_2]; W^{1,r_d})}^{\frac{d+2}{d-2}} \notag \\
&\quad+
\Big( \| u \|_{L^{q_d}([0,T_2]; W^{1,r_d})} + \| u_N \|_{L^{q_d}([0,T_2]; W^{1,r_d})} \Big)^{\frac{d+2}{d-2}}
\| u - u_N \|_{L^{q_d}([0,T_2]; W^{1,r_d})}\notag \\
& \too 0
\label{E4}
\end{align}

\noi
almost surely, 
as $N \to \infty$.

Let us now consider the second and fourth terms on  the right-hand side of \eqref{E3}.
As for the second term, we first  consider 
the contribution from $\cj{\N_1(u_N)} \phi_N e_n$.
By H\"older's inequality with \eqref{nonlin} and Sobolev's embedding: $\dot H^1(\R^d) \embeds L^{\frac{2d}{d-2}}(\R^d)$,
we have
\begin{align*}
\bigg|\int_{\R^d} \cj{\N_1(u_N)}  \phi_N e_n dx\bigg|
\le \| u_N \|_{L^{\frac{2d}{d-2}}}^{\frac{d+2}{d-2}} \| \phi_N e_n \|_{L^{\frac{2d}{d-2}}} 
\les \| u_N \|_{\dot H^1}^\frac{d+2}{d-2} \| \phi_N e_n \|_{\dot H^1}.
\end{align*}

\noi
Then, by Ito's isometry along with the independence of $\{\be_n\}_{n\in \NB}$, 
we obtain
\begin{align}
\E\Bigg[ \bigg|\sum_{n \in \NB} \int_0^t \int_{\R^d}
& \cj{ \N_1( u_N) (t')} \phi_N e_n dx   d \beta_n (t') \bigg|^2\Bigg]
\les t \sum_{n \in \NB} 
\| u_N \|_{L^\infty([0, t]; \dot H^1)}^{\frac{2(d+2)}{d-2}}
 \| \phi_N e_n \|_{\dot H^1}^2\notag \\
& \les t 
\| u_N \|_{L^\infty([0, t]; \dot H^1)}^{\frac{2(d+2)}{d-2}}
 \| \phi_N \|_{\HS(L^2; \dot H^1)}^2.
\label{E5}
\end{align}

\noi
By integration by parts (in $x$)
and Ito's isometry, we  bound 
the contribution from 
$\cj{ \Delta u_N } \phi_N e_n$  by 
\begin{align}
\E\Bigg[ \bigg|\sum_{n \in \NB} \int_0^t \int_{\R^d}
 \cj{ \Dl u_N (t')} \phi_N e_n dx   d \beta_n (t') \bigg|^2\Bigg]
& \les t 
\| u_N \|_{L^\infty([0, t]; \dot H^1)}^2
 \| \phi_N \|_{\HS(L^2; \dot H^1)}^2.
\label{E6}
\end{align}

\noi
As for the  fourth term on the right-hand side of \eqref{E3}, 
it follows from H\"older's and Sobolev's inequalities that 
\begin{align}
\sum_{n \in \NB} 
\int_{\R^d} |u_N|^{\frac{4}{d-2}} |\phi _Ne_n|^2 dx
& \le \| u_N \|_{L^{\frac{2d}{d-2}}}^{\frac{4}{d-2}}
\sum_{n \in \NB} 
 \| \phi_N e_n \|_{L^{\frac{2d}{d-2}}}^2\notag \\
& \les \| u_N \|_{\dot H^1}^{\frac{4}{d-2}} \| \phi_N  \|_{\HS(L^2; \dot H^1)}^2.
\label{E7}
\end{align}

\noi
Since
$3 \le d \le 6$, we have $\frac{4}{d-2} \ge 1$,
which implies that difference estimates 
on the contributions from $u_N$ and $u$ 
for \eqref{E5}, \eqref{E6}, and \eqref{E7}
also hold.
Therefore, by 
in view of~\eqref{E4}
and (the difference estimates for) 
\eqref{E5}, \eqref{E6}, and \eqref{E7}, 
we obtain 
 \eqref{E1}
 by taking 
 $N \to \infty$ in \eqref{E3} and then $R \to \infty$.
This concludes the proof of Lemma \ref{LEM:ito}.
\end{proof}

\begin{acknowledgment}

\rm 
T.O.~was supported by the European Research Council (grant no.~637995 ``ProbDynDispEq'').
M.O.~was supported by JSPS KAKENHI Grant number JP16K17624.
The authors would like to thank the anonymous referee for  helpful comments.
\end{acknowledgment}

\end{document}